\documentclass[oneside,french,american]{amsart}
\usepackage[T1]{fontenc}
\usepackage[latin9]{inputenc}
\usepackage{geometry}
\usepackage{xcolor}
\usepackage{pdfcolmk}
\usepackage{amsthm}
\usepackage{amssymb}
\usepackage[all]{xy}
\PassOptionsToPackage{normalem}{ulem}
\usepackage{ulem}

\makeatletter

\providecolor{lyxadded}{rgb}{0,0,1}
\providecolor{lyxdeleted}{rgb}{1,0,0}

\numberwithin{equation}{section}
\numberwithin{figure}{section}
  \theoremstyle{plain}
  \newtheorem*{thm*}{\protect\theoremname}
  \theoremstyle{definition}
  \newtheorem{defn}{\protect\definitionname}[section]
  \theoremstyle{plain}
  \newtheorem{thm}{\protect\theoremname}[section]
\newenvironment{lyxcode}
{\par\begin{list}{}{
\setlength{\rightmargin}{\leftmargin}
\setlength{\listparindent}{0pt}
\raggedright
\setlength{\itemsep}{0pt}
\setlength{\parsep}{0pt}
\normalfont\ttfamily}%
 \item[]}
{\end{list}}
  \theoremstyle{plain}
  \newtheorem{lem}{\protect\lemmaname}[section]
  \theoremstyle{plain}
  \newtheorem{cor}{\protect\corollaryname}[section]
  \theoremstyle{remark}
  \newtheorem{rem}{\protect\remarkname}[section]
  \theoremstyle{plain}
  \newtheorem{prop}{\protect\propositionname}[section]
  \theoremstyle{definition}
  \newtheorem{example}{\protect\examplename}[section]
  \theoremstyle{remark}
  \newtheorem*{rem*}{\protect\remarkname}

\makeatother

\usepackage{babel}
\makeatletter
\addto\extrasfrench{%
   \providecommand{\fg}{\ifdim\lastskip>\z@\unskip\fi~\frqq}%
}

\makeatother
  \addto\captionsamerican{\renewcommand{\definitionname}{Definition}}
  \addto\captionsamerican{\renewcommand{\examplename}{Example}}
  \addto\captionsamerican{\renewcommand{\lemmaname}{Lemma}}
  \addto\captionsamerican{\renewcommand{\propositionname}{Proposition}}
  \addto\captionsamerican{\renewcommand{\remarkname}{Remark}}
  \addto\captionsamerican{\renewcommand{\theoremname}{Theorem}}
  \addto\captionsfrench{\renewcommand{\definitionname}{Définition}}
  \addto\captionsfrench{\renewcommand{\examplename}{Exemple}}
  \addto\captionsfrench{\renewcommand{\lemmaname}{Lemme}}
  \addto\captionsfrench{\renewcommand{\propositionname}{Proposition}}
  \addto\captionsfrench{\renewcommand{\remarkname}{Remarque}}
  \addto\captionsfrench{\renewcommand{\theoremname}{Théorème}}
  \providecommand{\definitionname}{Definition}
  \providecommand{\examplename}{Example}
  \providecommand{\lemmaname}{Lemma}
  \providecommand{\propositionname}{Proposition}
  \providecommand{\remarkname}{Remark}
  \providecommand{\theoremname}{Theorem}
\addto\captionsamerican{\renewcommand{\corollaryname}{Corollary}}
\addto\captionsamerican{\renewcommand{\theoremname}{Theorem}}
\addto\captionsfrench{\renewcommand{\corollaryname}{Corollaire}}
\addto\captionsfrench{\renewcommand{\theoremname}{Théorème}}
\providecommand{\corollaryname}{Corollary}
\providecommand{\theoremname}{Theorem}

\begin{document}

\title{Cyclic covers of affine $\mathbb{T}$-varieties}



\author{Charlie Petitjean }

\address{Charlie Petitjean, Institut de
Mathématiques de Bourgogne, Université
de Bourgogne, 9 Avenue Alain Savary,
BP 47870, 21078 Dijon Cedex, France}

\email{charlie.petitjean@u-bourgogne.fr}
\begin{abstract}
We consider normal affine $\mathbb{T}$-varieties
$X$ endowed with an action of
finite abelian group $G$ commuting
with the action of $\mathbb{T}$.
For such varieties we establish
the existence of $G$-equivariant
geometrico-combinatorial presentations
in the sense of Altmann and Hausen.
As an application, we determine
explicit presentations of the Koras-Russell
threefolds as bi-cyclic covers
of $\mathbb{A}^{3}$ equipped with
a hyperbolic $\mathbb{C}^{*}$-action.
\end{abstract}
\maketitle

\section*{Introduction}

Every algebraic action of the one
dimensional torus $\mathbb{T}\simeq\mathbb{C}^{*}$
on a complex affine variety $X$
is determined by a $\mathbb{Z}$-grading
$A=\oplus_{m\in\mathbb{Z}}A_{m}$
of its coordinate ring $A$, the
spaces $A_{m}$ consisting of semi-invariant
regular functions of weight $m$
on $X$. One possible way to construct
$\mathbb{Z}$-graded algebras,
which was studied by Demazure \cite{D},
is to start with a variety $Y$
and a $\mathbb{Q}$-divisor $D$
on $Y$ and to let $A=\oplus_{m\in\mathbb{Z}}\Gamma(Y,\mathcal{O}_{Y}(mD))$.
For a well chosen pair $(Y,D)$,
this algebra is finitely generated,
corresponding to the ring of regular
functions of an affine variety,
$X=\mathbb{S}(Y,D)$ with a $\mathbb{C}^{*}$-action
whose algebraic quotient is isomorphic
to $\mathrm{Spec}(\Gamma(Y,\mathcal{O}_{Y}))$.
A slight variant of this construction
\cite{F-Z} already enabled a complete
description of $\mathbb{C}^{*}$-actions
on normal surfaces $X$: namely
they correspond to graded algebras
of the form: 
\[
A=\oplus_{m<0}\Gamma(Y,\mathcal{O}_{Y}(mD_{-}))\oplus\Gamma(Y,\mathcal{O}_{Y})\oplus_{m>0}\Gamma(Y,\mathcal{O}_{Y}(mD_{+})),
\]
 for suitably chosen triples $(Y,D_{+},D_{-})$
consisting of a smooth curve $Y$
and a pair of $\mathbb{Q}$-divisors
$D_{+}$ and $D_{-}$ on it.

Demazure's construction was generalized
by Altmann and Hausen \cite{A-H}
to give a description of all normal
affine varieties $X$ equipped
with an effective action of an
algebraic torus $\mathbb{T}\simeq(\mathbb{C}^{*})^{k}$,
$k\geq1$. Here the $\mathbb{Z}$-grading
is replaced by a grading by the
lattice $M\simeq\mathbb{Z}^{k}$
of characters of the torus, and
the graded pieces are recovered
from a datum consisting of a variety
$Y$ of dimension $\mathrm{dim}(X)-\mathrm{dim}(\mathbb{T})$
and a so-called polyhedral divisor
$\mathcal{D}$ on $Y$, a generalization
of $\mathbb{Q}$-divisors for higher
dimensional tori: $\mathcal{D}$
can be considered as a collection
of $\mathbb{Q}$-divisors $\mathcal{D}(u)$
parametrized by a ``weight cone
'' $\sigma^{\vee}\cap M$, for
which we have $A=\oplus_{u\in\sigma^{\vee}\cap M}\Gamma(Y,\mathcal{O}_{Y}(\mathcal{D}(u)))$.
The $\mathbb{T}$-variety associated
to a pair $(Y,\mathcal{D})$ is
denoted by $\mathbb{S}(Y,\mathcal{D})$. 

In this article, we consider affine
$\mathbb{T}$-varieties $X$ endowed
with an additional action of a
finite abelian group $G$ commuting
with the action of $\mathbb{T}$.
The quotient $X'=X/\!/G$ is again
an affine $\mathbb{T}'$-variety
for a torus $\mathbb{T}'\simeq\mathbb{T}$
obtained as a quotient of $\mathbb{T}$
by an apporpriate finite group,
and our aim is to understand the
relation between the presentations
$X=\mathbb{S}(Y,\mathcal{D})$
of $X$ and those of $X'=\mathbb{S}(Y',\mathcal{D}')$.
A pair $(Y,\mathcal{D})$ such
that $X=\mathbb{S}(Y,\mathcal{D})$
is not unique but we will show
that it is always possible to choose
a particular pair $(Y,\mathcal{D}_{G})$
consisting of a variety $Y$ endowed
with a $G$-action and a $G$-invariant
polyhedral divisor $\mathcal{D}_{G}$
such that $X$ is $G\times\mathbb{T}$
equivariantly isomorphic to $\mathbb{S}(Y,\mathcal{D}_{G})$.
The $G$-invariant divisor $\mathcal{D}_{G}$
corresponds in turn to a certain
polyhedral divisor $\mathcal{D}'$
on the quotient $Y/\!/G$ with
property that $X'=\mathbb{S}(Y/\!/G,\mathcal{D}')$
as a $\mathbb{T}'$-variety.

More precisely, our main result
reads as follows:
\begin{thm*}
Let $X$ be a $\mathbb{T}$-variety
and let $G$ be a finite abelian
group acting on $X$ such that
the two actions commute. Then the
following hold:

1) There exist a semi-projective
variety $Y$ endowed with an action
of $G$ and a $G$-invariant pp-divisor
$\mathcal{D}_{G}$ defined on $Y$
such that $X$ is $\mathbb{T}\times G$
equivariantly isomorphic to $\mathbb{S}(Y,\mathcal{D}_{G})$.

2) Moreover $X/\!/G$ is equivariantly
isomorphic to the $\mathbb{T}'$-variety
$\mathbb{S}(Y/\!/G,\mathcal{D}')$
where $\mathcal{D}'$ can be chosen
such that $F_{*}(\mathcal{D}_{G})=\varphi_{G}^{*}(\mathcal{D}')$,
where $\varphi_{G}:Y\rightarrow Y/\!/G$
denotes the quotient morphism and
$F:M^{\vee}\rightarrow M'^{\vee}$
is a linear map induced by the
inclusion between the character
lattices $M'$ of $\mathbb{T}'$
and $M$ of $\mathbb{T}$.
\end{thm*}
We then apply this result to determine
presentations of a family of exotic
affine spaces of dimension $3$
with hyperbolic $\mathbb{C}^{*}$-actions:
the Koras-Russell threefolds. We
exploit the fact that these threefolds
arise as equivariant bi-cyclic
cover of the affine space $\mathbb{A}^{3}$
equipped with a hyperbolic $\mathbb{C}^{*}$-action.

The article is organized as follows.
The first section is devoted to
a short recollection on Altmann-Hausen
representations, with a particular
focus on the methods to construct
pairs $(Y,\mathcal{D})$ corresponding
to a given graded algebra. The
main theorem above is then established
in section two. Finally, explicit
Altmann-Hausen representations
of the Koras-Russell threefolds
are determined in section three.

\section{Recollection on the Altmann-Hausen
representation}

In this section, we introduce the
correspondence between normal affine
$\mathbb{T}$-varieties $X$ and
pairs $(Y,\mathcal{D})$ composed
of a normal semi-projective variety
$Y$ and a so-called polyhedral
divisor $\mathcal{D}$ established
by Altmann-Hausen \cite{A-H}.
In particular, for a given $X$,
we summarize a construction of
a corresponding $Y$ and explain
a method to determine a possible
$\mathcal{D}$.

\subsection{Normal affine $\mathbb{T}$-varieties}

Let $N\simeq\mathbb{Z}^{k}$ be
a lattice of rank $k$ and let
$M=\mathrm{Hom}(N,\mathbb{Z})$
be its dual. A pointed convex polyhedral
cone $\sigma\subseteq N_{\mathbb{Q}}=N\otimes_{\mathbb{Z}}\mathbb{Q}$
is an intersection of finitely
many closed linear half spaces
in $N_{\mathbb{Q}}$ which does
not contain any line. Its dual:
\[
\sigma^{\vee}:=\{v\in M_{\mathbb{Q}}\mid\forall u\in\sigma\,\left\langle u,v\right\rangle \geq0\}\subseteq M_{\mathbb{Q}}=M\otimes_{\mathbb{Z}}\mathbb{Q},
\]
 consists of all linear forms on
$M_{\mathbb{Q}}$ that are non-negative
on $\sigma$. A polytope $\Pi\subset N_{\mathbb{Q}}$
is the convex hull of finitely
many points in $N_{\mathbb{Q}}$,
and a convex polyhedron $\Delta\subseteq N_{\mathbb{Q}}$
is the intersection of finitely
many closed affine half spaces
in $N_{\mathbb{Q}}$. Every polyhedron
admits a decomposition: $\Delta=\Pi_{\Delta}+\sigma$,
where $\Pi_{\Delta}$ is a polytope
and $\sigma$ is a pointed convex
polyhedral cone, called the tail
cone of $\Delta$. The set of all
polyhedra which admit the same
tail cone is a semigroup with Minkowski
addition, which we denote by $\mathrm{Pol}_{\sigma}^{+}(N_{\mathbb{Q}})$.
\begin{defn}
A $\sigma$-\emph{tailed polyhedral
divisor} $\mathcal{D}$ on an algebraic
variety $Y$ is a formal finite
sum

\[
\mathcal{D}=\sum\Delta_{i}\otimes D_{i}\in\mathrm{Pol}_{\sigma}^{+}(N_{\mathbb{Q}})\otimes_{\mathbb{Z}}\mathrm{WDiv}(Y),
\]

\noindent \begin{flushleft}
where $D_{i}$ are prime divisors
on $Y$ and $\Delta_{i}$ are $\sigma$-polyhedra.
\par\end{flushleft}
\end{defn}
Every element $u\in\sigma^{\vee}\cap M$
determines a map $\mathrm{Pol}_{\sigma}^{+}(N_{\mathbb{Q}})\otimes_{\mathbb{Z}}\mathrm{WDiv}(Y)\rightarrow\mathbb{Q}\otimes_{\mathbb{Z}}\mathrm{WDiv}(Y)$
which associates to $\mathcal{D}=\sum\Delta_{i}\otimes D_{i}$
the Weil $\mathbb{Q}$-divisor
${\displaystyle \mathcal{D}(u)=\sum\underset{v\in\Delta_{i}}{\mathrm{min}}\left\langle u,v\right\rangle D_{i}}$
on $Y$. 

Given a Weil $\mathbb{Q}$-divisor
$D$ and a section $s\in\Gamma(Y,\mathcal{O}_{Y}(D))$,
that is, an effective Weil divisor
$D'$ linearly equivalent to the
round-down $\lfloor D\rfloor$
of $D$, we denote by $Y_{s}$
the open subset $Y\setminus\mathrm{Supp}(D')$
of $Y$ .
\begin{defn}
(\cite[Definition 2.5 and 2.7]{A-H})
A \emph{proper-polyhedral divisor},
noted pp-divisor, is a polyhedral
divisor $\mathcal{D}=\sum\Delta_{i}\otimes D_{i}$
on $Y$ which satisfies the following
properties:

1) Each $D_{i}$ is an effective
divisor and $\mathcal{D}(u)$ is
a $\mathbb{Q}$-Cartier divisor
on $Y$ for every $u\in\sigma^{\vee}\cap M$
.

2) $\mathcal{D}(u)$ is semi-ample
for each $u\in\sigma^{\vee}\cap M$,
that is, for some $n\in\mathbb{Z}_{>0}$
the open subsets $Y_{s}$ , where
$s\in\Gamma(Y,\mathcal{O}_{Y}(\mathcal{D}(nu)))$,
cover $Y$. 

3) $\mathcal{D}(u)$ is big for
each $u\in\mathrm{relint}(\sigma^{\vee})\cap M$,
that is, for some $n\in\mathbb{Z}_{>0}$
there exist a section $s\in\Gamma(Y,\mathcal{O}_{Y}(\mathcal{D}(nu)))$
such that $Y_{s}$ is affine. 
\end{defn}
Recall \cite[Definition 2.1]{A-H}
that a variety $Y$ is said to
be semi-projective if $\Gamma(Y,\mathcal{O}_{Y})$
is finitely generated and $Y$
is projective over $Y_{0}=\mathrm{Spec}(\Gamma(Y,\mathcal{O}_{Y}))$.
Given a pp-divisor $\mathcal{D}$
on $Y$, the graded algebra 
\[
A=\underset{u\in\sigma^{\vee}\cap M}{\bigoplus}A_{u}=\underset{u\in\sigma^{\vee}\cap M}{\bigoplus}\Gamma(Y,\mathcal{O}_{Y}(\mathcal{D}(u))).
\]
 is finitely generated, and $\mathrm{Spec}(A)$
is a $\mathbb{T}$-variety for
$\mathbb{T}=\mathrm{Spec}(\mathbb{C}[M])\simeq(\mathbb{C}^{*})^{k}$.
More precisely Altmann and Hausen,
showed the following:
\begin{thm}
\cite{A-H} For any pp-divisor
$\mathcal{D}$ on a normal semi-projective
variety $Y$, the scheme

\[
\mathbb{S}(Y,\mathcal{D})=\mathrm{Spec}(\underset{u\in\sigma^{\vee}\cap M}{\bigoplus}\Gamma(Y,\mathcal{O}_{Y}(\mathcal{D}(u))))
\]
 is a normal affine $\mathbb{T}$-variety
of dimension $\mathrm{dim}(Y)+\mathrm{dim}(\mathbb{T})$.
Conversely any normal affine $\mathbb{T}$-variety
is isomorphic to an $\mathbb{S}(Y,\mathcal{D})$
for suitable $Y$ and $\mathcal{D}$.
\end{thm}

\subsection{\label{-the-semi-projective}Determining
the semi-projective variety}

The semi-projective variety $Y$
is not unique, however there exists
a natural construction, which we
will use in the remainder of the
article. It can be summarized as
follows (\cite[section 6]{A-H}).
\\

Let $X=\mathrm{Spec}(\underset{u\in M}{\bigoplus}A_{u})$
be an affine variety endowed with
an effective action of the torus
$\mathbb{T}=\mathrm{Spec}(\mathbb{C}[M])$.
For each $u\in M$ the set of semistable
points 
\[
X^{ss}(u):=\left\{ x\in X\:/\:\exists n\in\mathbb{Z}_{\geq0}\: and\: f\in A_{nu}\: such\: that\: f(x)\neq0\right\} 
\]

\noindent \begin{flushleft}
is an open $\mathbb{T}$-invariant
subset of $X$ which admits a good
$\mathbb{T}$-quotient
\par\end{flushleft}

\[
Y_{u}=X^{ss}(u)/\!/\mathbb{T}=\mathrm{Proj}_{A_{0}}(\underset{n\in\mathbb{Z}_{\geq0}}{\bigoplus}A_{nu}).
\]

Following \cite[section 6]{A-H},
there exists a fan $\Lambda\in M_{\mathbb{Q}}$
generated by a finite collection
of cones $\lambda$ such that the
following holds:

1) For any $u$ and $u'$ in the
relative interior of $\lambda$,
$X^{ss}(u)=X^{ss}(u')$. We denote
$W_{\lambda}=X^{ss}(u)$ for any
$u\in relint(\lambda)$ 

2) If $\gamma$ is a face of $\lambda$,
$W_{\lambda}$ is an open subset
of $W_{\gamma}$. Let $W=\underset{\lambda\in\Lambda}{\cap}W_{\lambda}=\underset{\leftarrow}{lim}W_{\lambda}$.

The quotient maps $q_{\lambda}:W_{\lambda}\rightarrow W_{\lambda}/\!/\mathbb{T}$
form an inverse system indexed
by the cones in $\Lambda$, whose
inverse limit exist as a morphism
$q:W\longrightarrow Z=\underset{\leftarrow}{lim}Y_{\lambda}$.
The desired semi-projective variety
$Y$ is the normalization of the
closure of the image of $W$ by
$q$. 
\selectlanguage{french}%
\begin{lyxcode}
\[
\xymatrix{W\ar[r]\ar[d]_{q} & W_{\lambda}\ar[r]\ar[d]_{q_{\lambda}} & W_{\gamma}\ar[r]\ar[d]_{q_{\gamma}} & X\ar[dd]_{q_{0}}\\
Z\ar[drrr]\ar[r] & Y_{\lambda}\ar[drr]\ar[r] & Y_{\gamma}\ar[dr]\\
 &  &  & Y_{0}=\mathrm{Spec}(A_{0})
}
\]

\end{lyxcode}
\selectlanguage{american}%

\subsection{\label{sub:Maps-of-proper}Maps
of proper polyhedral divisor.}

Let $Y$ and $Y'$ be normal semi-projective
varieties, $N$ and $N'$ be lattices
and $\sigma\subset N_{\mathbb{Q}}$,
$\sigma'\subset N'_{\mathbb{Q}}$
be pointed cones. Let $\mathcal{D}=\sum\Delta_{i}\otimes D_{i}$
and $\mathcal{D}'=\sum\Delta'_{i}\otimes D'_{i}$
be pp-divisors on $Y$ and $Y'$
respectively with corresponding
tail cones $\sigma$ and $\sigma'$.
\begin{defn}
\cite[Definition 8.3 ]{A-H}1)
For a morphism $\varphi:Y\rightarrow Y'$
such that $\varphi(Y)$ is not
contained in $\mathrm{Supp}(D'_{i})$
for any $i$, the polyhedral pull-back
of $\mathcal{D}'$ is defined by
:

\[
\varphi^{*}(\mathcal{D}'):=\sum\Delta'_{i}\otimes\varphi^{*}(D'_{i})
\]

\noindent \begin{flushleft}
Where $\varphi^{*}(D'_{i})$ is
the usual pull-back of $D'_{i}$.
It is a polyhedral divisor on $Y$
with tail cone $\sigma'$.
\par\end{flushleft}

2) For a linear map $F:N\rightarrow N'$
such that $F(\sigma)\subset\sigma'$,
the polyhedral push forward is
defined as :

\[
F_{*}(\mathcal{D}):=\sum(F(\Delta_{i})+\sigma')\otimes D_{i}
\]

\noindent \begin{flushleft}
It is also a polyhedral divisor
on $Y$ with tail cone $\sigma'$.
\par\end{flushleft}
\end{defn}
An equivariant morphism from $\mathbb{S}(Y,\mathcal{D})$
to $\mathbb{S}(Y',\mathcal{D}')$
is given by a homomorphism of algebraic
groups $\psi:\mathbb{T}\rightarrow\mathbb{T}'$
and a morphism $\phi:\mathbb{S}(Y,\mathcal{D})\rightarrow\mathbb{S}(Y',\mathcal{D}')$
satisfying $\phi(\lambda.x)=\psi(\lambda).\phi(x)$.
Every such morphism is uniquely
determined by a triple $(\varphi,F,f)$
defined as above consisting of
a dominant morphism $\varphi:Y\rightarrow Y'$,
a linear map $F:N\rightarrow N'$
as above and a \emph{plurifunction}
$f\in N'\otimes_{\mathbb{Z}}\mathbb{C}(Y)^{*}$
such that :

\[
\varphi^{*}(\mathcal{D}')\leq F_{*}(\mathcal{D})+\mathrm{div}(f).
\]
The identity map of a pp-divisor
is the triple $(\mathrm{id},\mathrm{id},1)$
and the composition of two maps
$(\varphi,F,f)$ and $(\varphi',F',f')$
is $(\varphi'\circ\varphi,F'\circ F,F'_{*}(f).\varphi^{*}(f'))$.

\subsection{\label{sub:Determining-proper-polyhedral}Determining
proper polyhedral divisors}

A method to determine a possible
pp-divisor $\mathcal{D}$ (\cite[section 11]{A-H})
associated to a $\mathbb{T}$-variety
$X$ with $\mathbb{T}=(\mathbb{C}^{*})^{k}$
is to embed $X$ as a $\mathbb{T}$-stable
subvariety of a toric variety.
The calculation is then reduced
to the toric case by considering
an embedding in $\mathbb{A}^{m}$
with linear action for $m$ sufficiently
large. In other words, $X$ is
realized as a $(\mathbb{C}^{*})^{k}$-stable
subvariety of a $(\mathbb{C}^{*})^{m}$-toric
variety. The inclusion of $(\mathbb{C}^{*})^{k}$
corresponds to an inclusion of
the lattice of characters $\mathbb{Z}^{k}$
of $\mathbb{T}$ into $\mathbb{Z}^{m}$.
We obtain the exact sequence: 
\[
\xymatrix{0\ar[r] & \mathbb{Z}^{k}\ar[r]_{F} & \mathbb{Z}^{m}\ar[r]_{P}\ar@/_{1pc}/[l]_{s} & \mathbb{Z}^{m}/\mathbb{Z}^{k}\ar[r] & 0}
,
\]
 where $F$ is given by the action
of $(\mathbb{C}^{*})^{k}$ on $\mathbb{A}^{m}$
and $s$ is a section of $F$.
The $(\mathbb{C}^{*})^{m}$-toric
variety is determined by the first
integral vectors $v_{i}$ of the
unidimensional cone generated by
the i-th column vector of $P$
as rays in a $\mathbb{Z}^{m}$
lattice, and each $v_{i}$ correspond
to a divisor. The support of $D_{i}$
is the intersection between $X$
and the divisor corresponding to
$v_{i}$. The tail cone is $\sigma:=s(\mathbb{Q}_{\geq0}^{m}\cap F(\mathbb{Q}))$,
and the polytopes are $\Pi_{i}=s(\mathbb{R}_{\geq0}^{m}\cap P^{-1}(v_{i}))$.

\section{Actions of finite abelian groups}

Let $X=\mathrm{Spec}(A)$ be a
normal affine variety with an effective
action of a torus $\mathbb{T}$
and let $G$ be a finite abelian
group of order $d\geq2$ whose
action on $X$ commutes with that
of $\mathbb{T}$. The goal of this
section is to determine the relationship
between the Altmann-Hausen representations
of $X$ and those of $X/\!/G=\mathrm{Spec}(A^{G})$.
\\

Let $Y$ be a semi-projective variety
equipped with an action $\psi:G\times Y\overset{}{\rightarrow}Y$
of an algebraic group $G.$ If
$\mathcal{D}_{G}$ is a $G$-invariant
pp-divisor, i.e a pp-divisor such
that $\psi(g,\cdot)^{*}\mathcal{D}_{G}=\mathcal{D}_{G}$
for every $g\in G$, then for every
$u\in\sigma^{\vee}\cap M$ the
space $A_{u}=\Gamma(Y,\mathcal{O}(\mathcal{D}_{G}(u)))$
of global sections $\mathcal{O}(\mathcal{D}_{G}(u))$
is endowed with a $G$-action.
It follows that $\mathbb{S}(Y,\mathcal{D}_{G})=\mathrm{Spec}(\underset{u\in\sigma^{\vee}\cap M}{\bigoplus}\Gamma(Y,\mathcal{O}_{Y}(\mathcal{D}(u))))$
admits an action of $G$ commuting
with that of $\mathbb{T}$.
\begin{thm}
\label{thm:Main_TH} Let $X$ be
a $\mathbb{T}$-variety and let
$G$ be a finite abelian group
acting on $X$ such that the two
actions commute. Then the following
hold: 

1) There exist a semi-projective
variety $Y$ endowed with an action
of $G$ and a $G$-invariant pp-divisor
$\mathcal{D}_{G}$ on $Y$ such
that $X$ is $\mathbb{T}\times G$
equivariantly isomorphic to $\mathbb{S}(Y,\mathcal{D}_{G})$.

2) Moreover $X/\!/G$ is equivariantly
isomorphic to the $\mathbb{T}'$-variety
$\mathbb{S}(Y/\!/G,\mathcal{D}')$
where $\mathcal{D}'$ can be chosen
such that $F_{*}(\mathcal{D}_{G})=\varphi_{G}^{*}(\mathcal{D}')$,
where $\varphi_{G}:Y\rightarrow Y/\!/G$
denotes the quotient morphism and
$F:M^{\vee}\rightarrow M'^{\vee}$
is a linear map induced by the
inclusion between the character
lattices $M'$ of $\mathbb{T}'$
and $M$ of $\mathbb{T}$ (see
\ref{G in T} ).
\end{thm}
We will divide the proof in several
steps. First we will prove that
the action of $G$ on $X$ induces
an action of $G$ on $Y$. Secondly
we will consider the case where
the orbits of the $G$-action are
included in the orbits of the $\mathbb{T}$-action
and finally we consider the case
where the action of $G\times\mathbb{T}$
is effective on $X$.
\begin{lem}
\label{normalization}Let $Y$
a quasi-projective variety endowed
with an action of a finite group
$G$ and let $\widehat{Y}\rightarrow Y$
be the normalization of $Y$. Then
the action of $G$ lifts to an
action on $\widehat{Y}$ and the
induced morphism $\widehat{Y}/\!/G\rightarrow Y/\!/G$
is the normalization of $Y/\!/G$.\end{lem}
\begin{proof}
Since $Y$ is quasi-projective
and $G$ is finite, every $x\in X$
admits a $G$-invariant affine
open neighborhood. The normalization
being a local operation, we may
assume that $Y$ is affine. Using
the universal properties of the
normalization and of the quotient
we obtain the following commutative
diagram:

\[
\begin{array}{ccc}
Y & \rightarrow & Y/\!/G\\
\uparrow &  & \uparrow\\
\widehat{Y} & \rightarrow & \widehat{Y/\!/G}\\
 & \searrow & \uparrow\\
 &  & \widehat{Y}/\!/G
\end{array}
\]
Thus $\mathbb{C}[\widehat{Y/\!/G}]\subset\mathbb{C}[\widehat{Y}]^{G}$.
Conversely, let $f\in\mathbb{C}[\widehat{Y}]^{G}$.
Then $g.f=f$ for all $g$ in $G$
and there exists a monic polynomial
$P$ with coefficients in $\mathbb{C}[Y]$
such that $P(f)=0$. Since $G$
is finite, $Q=\underset{g\in G}{\prod}g.P$
is a monic polynomial with $G$-invariant
coefficients and $G(f)=0$. So
$f\in\mathbb{C}[\widehat{Y/\!/G}]$
.\end{proof}
\begin{cor}
\label{G Y}Let $X$ be a $\mathbb{T}$-variety
and suppose that a finite abelian
group $G$ acts on $X$ such the
two actions commute. Then there
exists a semi-projective variety
$Y$ and a pp-divisor $\mathcal{D}$
on $Y$ such that $X$ is $G\times\mathbb{T}$
equivariantly isomorphic to $\mathbb{S}(Y,\mathcal{D})$
and the action of $G$ on $\mathbb{S}(Y,\mathcal{D})$
induces an action of $G$ on $Y$.\end{cor}
\begin{proof}
We consider the construction of
$Y$ given in section \ref{-the-semi-projective}.
Since the action of $G$ and $\mathbb{T}$
commute, for every $\lambda\in\Lambda$
the subset $X^{ss}(u)$ with $u\in relint(\lambda)$
is $G$-stable. Thus $W:=\cap_{\lambda\in\Lambda}W_{\lambda}$
is also $G$-stable. Since $q':W\rightarrow Z$
is the quotient by $\mathbb{T}$,
the action of $G$ on $W$ induces
one on $q'(W)$. The closure $\overline{q'(W)}$
is again $G$-stable, and since
$\overline{q'(W)}$ is quasi-projective
it follows from lemma \ref{normalization}
that the action of $G$ lifts to
an action on $Y$.\end{proof}
\begin{lem}
\label{G inv divisor}Let $X=\mathrm{Spec}(A)$
be a $\mathbb{T}$-variety and
let $G$ a finite abelian group
acting on $X$ such that the two
actions commute. Then there exists
a $G$-invariant pp-divisor $\mathcal{D}_{G}$
defined on $Y$ such that $X$
is equivariantly isomorphic to
$\mathbb{S}(Y,\mathcal{D}_{G})$.\end{lem}
\begin{proof}
By lemma \ref{G Y} the action
of $G$ on $X$ induces an action
of $G$ on $Y$. By the proof of
Theorem 3.4 in \cite{A-H}, a pp-divisor
on $Y$ corresponding to $X$ is
determined by the choice of a homomorphism
$h$ from $M$ into the fraction
field of $A$ with the property
that for every $u\in M$, $h(u)$
is semi-invariant of weight $u$.
Namely, if $u\in\sigma^{\vee}\cap M$
is any \emph{saturated element},
that is, $u\in\sigma^{\vee}\cap M$
such that $\bigoplus_{n\in\mathbb{N}}A_{nu}$
is generated in degree $1$, then
there exist a unique Cartier divisor
$\mathcal{D}(u)$ such that $A_{u}=h(u).\Gamma(Y,\mathcal{O}_{Y}(\mathcal{D}(u)))$:
its local equations on open subsets
$Y_{s}$ with $s\in A_{u}$ are
$h(u)/s$. By definition $h(u)=\frac{f}{g}$
where $f$ and $g$ are both non
zero and $f\in A_{u_{1}}$, $g\in A_{u_{2}}$
such that $u_{1}-u_{2}=u$. Since
$A_{u}$ is $G$-stable for all
$u\in M$, we can choose $f\in A_{u_{1}}$,
$g\in A_{u_{2}}$ semi-invariant
for the action of $G$ with $u_{1}-u_{2}=u$
so that $h(u)=f/g$ is also semi-invariant
for $G$. The corresponding divisor
$\mathcal{D}(u)$ is then $G$-invariant.
In the case of a general $u\in\sigma^{\vee}\cap M$,
we can choose a saturated multiple
$nu$ and define $\mathcal{D}(u)=\mathcal{D}(nu)/n$.
\end{proof}
To complete the proof of Theorem
\ref{thm:Main_TH}, we divide the
argument into two cases. First
we consider the situation where
$G$ is a subgroup of $\mathbb{T}$
and secondly where the action of
$G\times\mathbb{T}$ is effective. 
\begin{lem}
\label{G in T}Let $X$ be the
$\mathbb{T}$-variety $\mathbb{S}(Y,\mathcal{D})$
and let $G$ be a finite abelian
subgroup of $\mathbb{T}=\mathrm{Spec}(\mathbb{C}[M])$.
Then $X'=X/\!/G$ is a $\mathbb{T}'$-variety
where $\mathbb{T}'\simeq\mathbb{T}/\!/G$
and is equivariantly isomorphic
to the $\mathbb{S}(Y,F_{*}(\mathcal{D}))$
where $F:N=M^{\vee}\rightarrow N'=(M')^{\vee}$
is the linear map induced by the
inclusion between the character
lattices $M'$ of $\mathbb{T}'$
and $M$ of $\mathbb{T}$.\end{lem}
\begin{proof}
Let $Y$ be as in \ref{-the-semi-projective}.
Since by hypothesis the $G$-orbits
are contained in $\mathbb{T}$-orbits,
the induced $G$-action on $Y$
is trivial. In this case, for each
$u\in\sigma^{\vee}\cap M$, $A_{u}^{G}$
is either $A_{u}$ or $\{0\}$.
Letting $M'$ be the sublattice
$M$ generated by the elements
$u\in\sigma^{\vee}\cap M$ such
that $A_{u}^{G}\neq0$, 

\[
X'=X/\!/G=\mathrm{Spec}(\underset{u\in\sigma^{\vee}\cap M'}{\bigoplus}A_{u}^{G})
\]
is a $\mathbb{T}'$-variety where
$\mathbb{T}'=\mathrm{Spec}(\mathbb{C}[M'])$
is a torus of the same dimension
as $\mathbb{T}$. The inclusion
$M'\hookrightarrow M$ gives rise
the desired linear map $F:N=M^{\vee}\rightarrow N'=M'^{\vee}$.\end{proof}
\begin{rem}
This case corresponds to the map
of pp-divisors $(\mathrm{id},F,1)$
defined in as \ref{sub:Maps-of-proper}.
Indeed the quotient morphism $\varphi:Y\rightarrow Y/\!/G$
is the identity.\end{rem}
\begin{lem}
\label{Effective}Let $X$ be a
normal affine variety with an effective
action of $G\times\mathbb{T}$
where $G$ is a finite abelian
group. Then there exists a semi-projective
variety $Y$ on which $G$ acts
and a $G$-invariant pp-divisor
$\mathcal{D}_{G}$ on $Y$ such
that $X$ is $G\times\mathbb{T}$-equivariantly
isomorphic to $\mathbb{S}(Y,\mathcal{D}_{G})$.

Moreover $X/\!/G$ is $\mathbb{T}$-equivariantly
isomorphic to $\mathbb{S}(Y/\!/G,\mathcal{D}')$
where $\mathcal{D}_{G}=\varphi_{G}^{*}(\mathcal{D}')$.\end{lem}
\begin{proof}
By lemmas \ref{G Y} and \ref{G inv divisor},
$Y$ is endowed with an action
of $G$, and we can assume that
$X$ is equivariantly isomorphic
to $\mathbb{S}(Y,\mathcal{D}_{G})$.
Since $\mathcal{D}_{G}$ is $G$-stable,
for each $u\in\sigma^{\vee}\cap M$
, $\Gamma(Y,\mathcal{O}(\mathcal{D}_{G}(u)))$
is a $G$-invariant submodule of
$\Gamma(X,\mathcal{O}_{X})$ and
moreover there exists $\mathcal{D}'$
satisfies $\varphi_{G}^{*}\left(\mathcal{D}'\right)=\mathcal{D}_{G}$.
Therefore, $\Gamma(X/\!/G,\mathcal{O}_{X/\!/G})=(\underset{u\in\sigma^{\vee}\cap M}{\bigoplus}\Gamma(Y,\mathcal{O}(\mathcal{D}_{G}(u))))^{G}=\underset{u\in\sigma^{\vee}\cap M}{\bigoplus}\Gamma(Y,\mathcal{O}(\mathcal{D}_{G}(u)))^{G}$.

By assumption, $\varphi:Y\rightarrow Y/\!/G$
is the quotient morphism, and $\mathcal{D}'$
satisfying $\varphi_{G}^{*}\left(\mathcal{D}'\right)=\mathcal{D}_{G}$.
Thus 
\begin{eqnarray*}
\Gamma(Y,\mathcal{O}(\mathcal{D}_{G}(u)))^{G} & = & \left\{ f\in\mathbb{C}(Y)^{G},div(f)+\mathcal{D}_{G}(u)\geq0\right\} \cup\left\{ 0\right\} \\
 & = & \left\{ h\in\mathbb{C}(Y/\!/G),\varphi^{*}(\mathrm{div}(h)+\mathcal{D}'(u))\geq0\right\} \cup\left\{ 0\right\} \\
 & = & \left\{ h\in\mathbb{C}(Y/\!/G),\mathrm{div}(h)+\mathcal{D}'(u)\geq0\right\} \cup\left\{ 0\right\} .
\end{eqnarray*}
We conclude that $X/\!/G\simeq\mathrm{Spec}(\underset{u\in\sigma^{\vee}\cap M}{\bigoplus}\Gamma(Y/\!/G,\mathcal{O}(\mathcal{D}'(u))))$.\end{proof}
\begin{rem}
This lemma is the analogue of 4.1
in \cite{D}, in which Demazure
established a similar result for
algebras constructed from $\mathbb{Q}$-divisors.
This case corresponds to the map
of proper polyhedral divisors $(\varphi_{G},\mathrm{id},1)$
defined as in \ref{sub:Maps-of-proper}.\end{rem}
\begin{proof}
(of Theorem \ref{thm:Main_TH})
Consider a finite abelian group
$G$ acting on $X=\mathbb{S}(Y,\mathcal{D})$
whose action commutes with that
of $\mathbb{T}$. By virtue of
lemmas \ref{G Y} and \ref{G inv divisor},
we may assume that $G$ acts on
$Y$ and that $\mathcal{D}$ is
$G$-invariant. Then we let $H$
be the subgroup of $G\times\mathbb{T}$
consisting of elements which act
trivially on $X$. We let $G_{0}\subset G$
and $\mathbb{T}_{0}\subset\mathbb{T}$
be the images of $H$ by the two
projections and we let $G'=G/G_{0}$
and $\mathbb{T}'=\mathbb{T}/\mathbb{T}_{0}$.
Applying lemma \ref{G in T} to
$X$ equipped with the action of
$G_{0}$, we obtain a variety $X/\!/G_{0}$
endowed with an effective action
of $G'\times\mathbb{T}'$ to which
the lemma \ref{Effective} can
be applied. Any map $(\varphi_{G},F,1)$
is obtained by composing maps of
the two types above. 
\end{proof}

\section{Applications in the case $\mathbb{T}=\mathbb{C}^{*}$}

\subsection{Basic examples of $\mathbb{C}^{*}$-actions}

The coordinate ring of a normal
affine variety $X=\mathrm{Spec}(A)$
equipped with an effective $\mathbb{C}^{*}$-action
is $\mathbb{Z}$-graded in a natural
way via $A=\bigoplus_{n\in\mathbb{Z}}A_{n}$
where $A_{n}:=\left\{ f\in A/f(\lambda\cdot x)=\lambda^{n}f(x)\right\} $.
The semi-projective variety associated
to the Altmann-Hausen representation
of $X$ is the irreducible component
which correspond to the normalization
of the closure of the image of
$W$ by $q'$ (see \ref{-the-semi-projective})
in the fiber product :

\[
Y(X):=Y_{-}(X)\underset{Y_{0}(X)}{\times}Y_{+}(X)
\]

\noindent \begin{flushleft}
where $Y_{0}(X)=X/\!/\mathbb{C}^{*}=\mathrm{Spec}(A_{0})$
, $Y_{\pm}(X)=\mathrm{Proj}_{A_{0}}(\underset{n\in\mathbb{Z}_{\geq0}}{\bigoplus}A_{\pm n})$.
\par\end{flushleft}

A $\mathbb{C}^{*}$-action said
to be \emph{hyperbolic} if there
is at least one $n_{1}<0$ and
one $n_{2}>0$ such that $A_{n_{1}}$
and $A_{n_{2}}$ are nonzero. In
this case, the tail cone $\sigma$
is equal to $\{0\}$ (see \ref{sub:Determining-proper-polyhedral}
). If in addition $X$ is smooth,
then $Y(X)$ is in fact equal to
the fiber product which is itself
isomorphic to the blow-up of $Y_{0}(X)$
with center at th closed suscheme
defined by the ideal $\mathcal{I}=\left\langle A_{d}.A_{-d}\right\rangle $
where $d>0$ is chosen so that
$\bigoplus_{n\in\mathbb{Z}}A_{dn}$
is generated by $A_{0}$ and $A_{\pm d}$
( \cite{T} Theorem 1.9 and proposition
1.4).

In what follows, we denote by $\pi:\tilde{\mathbb{A}}_{(I)}^{n}\rightarrow\mathbb{A}^{n}$
the blow-up of the ideal $(I)$
in $\mathbb{A}_{(x_{1},...,x_{n})}^{n}=\mathrm{Spec}(\mathbb{C}[x_{1},...,x_{n}])$
.\\

Given an irreducible and reduced
hypersurface $H=\left\{ f(x_{1},...,x_{n})=0\right\} \subset\mathbb{A}^{n}$
containing the origin, the hypersurface
$X_{n,p,f}$ of $\mathbb{A}^{n+2}=\mathrm{Spec}(\mathbb{C}[x_{1},...,x_{n}][y,t])$
defined by the equation 
\[
\frac{f(x_{1}y,...,x_{n}y)}{y}+t^{p}=0
\]
comes equipped with an effective
$\mathbb{C}^{*}$-action induced
by the linear one $\lambda\cdot(x_{1},...,x_{n},y,t)=(\lambda^{p}x_{1},...,\lambda^{p}x_{n},\lambda^{-p}y,\lambda t)$
on $\mathbb{A}^{n+2}$. We have
$\mathbb{A}^{n+2}/\!/\mathbb{C}^{*}\simeq\mathbb{A}^{n+1}=\mathrm{Spec}(\mathbb{C}[u_{1},...,u_{n+1}])$
via $u_{i}=x_{i}y$ for $i=1,...,n$
and $u_{n+1}=yt^{p}$.
\begin{prop}
\label{building block} The variety
$X_{n,p,f}$ is equivariantly isomorphic
to $\mathbb{S}(\tilde{\mathbb{A}}_{(u_{1},...,u_{n})}^{n},\mathcal{D})$
for with $\mathcal{D}=\left\{ \frac{1}{p}\right\} D+[0,\frac{1}{p}]E$,
where $E$ is the exceptional divisor
of the blow up and $D$ is the
strict transform of the hypersurface
$H\subset\mathbb{A}^{n}$. \end{prop}
\begin{proof}
We determine $Y(X_{n,p,f})$ and
the pp-divisor $\mathcal{D}$ using
the method described in sections
1.2 and 1.3. We consider the exact
sequence :

\[
\xymatrix{0\ar[r] & \mathbb{Z}\ar[r]_{F} & \mathbb{Z}^{n+2}\ar[r]_{P}\ar@/_{1pc}/[l]_{s} & \mathbb{Z}^{n+1}\ar[r] & 0}
\]

\noindent \begin{flushleft}
where $F=\overset{}{}^{t}(p,...,p,-p,1)$,
$P=\left(\begin{array}{ccccc}
 &  &  & 1 & 0\\
 & \mathrm{I}_{n} &  & \vdots & \vdots\\
 &  &  & 1 & 0\\
0 & \cdots & 0 & 1 & p
\end{array}\right)$ $\mathrm{I}_{n}$ being the identity
matrix of rank $n\times n$ and
$s=(0,...,0,1)$.
\par\end{flushleft}

The fan in $\mathbb{Z}^{n+2}$
is generated by the rays $\left\{ v_{i}\right\} _{i=1,..,n+2}$
where $v_{i}$ is the first integral
vector of the unidimensional cone
generated by the i-th column vector
of $P$. It corresponds to the
blow up of the origin in $\mathbb{A}^{n+1}$,
as a toric variety.

The variety $Y$ is equal to the
strict transform by $\pi:\tilde{\mathbb{A}}_{(u_{1},...,u_{n})}^{n+1}\rightarrow\mathbb{A}^{n+1}\simeq\mathbb{A}^{n+2}/\!/\mathbb{C}^{*}$
of ${\displaystyle {\displaystyle \left\{ f(u_{1},...,u_{n})+u_{n+1}=0\right\} \subset\mathbb{A}^{n+1}}}$,
thus $Y\simeq\tilde{\mathbb{A}}^{n}$.

Since $\sigma:=s(\mathbb{Q}_{\geq0}^{m}\cap F(\mathbb{Q}))$
is $\{0\}$, applying the formula
$\Pi_{i}=s(\mathbb{R}_{\geq0}^{m}\cap P^{-1}(v_{i}))$,
we deduce that $\mathcal{D}$ has
the form $\left\{ \frac{1}{p}\right\} D+[0,\frac{1}{p}]E$,
where $D$ corresponds to the restriction
to $Y$ of the toric divisor given
by the ray $v_{n+2}$. It is the
restriction of $\left\{ u_{n+1}=yt^{p}=0\right\} $
to $Y$ thus $D$ is the strict
transforms of the hypersurface
$H\subset\mathbb{A}^{n}$. The
divisor $E$ corresponds to the
restriction to $Y$ of the toric
divisor given by $v_{n+1}$, that
is, the exceptional divisor.\end{proof}
\begin{example}
\label{3.1}Specializing the above
construction we obtain examples
of linear hyperbolic $\mathbb{C}^{*}$-actions
on $\mathbb{A}^{3}$ which will
be building blocks for further
applications :

a) Choosing $n=2$ and $f(x_{1},x_{2})=x_{1}$,
we obtain that $X_{2,x_{1},p}$
is isomorphic to $\mathbb{S}(\tilde{\mathbb{A}}_{(u,v)}^{2},\mathcal{D})$
with $\mathcal{D}=\left\{ \frac{1}{p}\right\} D+[0,\frac{1}{p}]E$,
where $E$ is the exceptional divisor
of the blow up and $D$ is the
strict transform of the line $\{u=0\}\subset\mathbb{A}^{2}$.
Thus $X_{2,x_{1},p}\subset\mathbb{A}^{4}$
is isomorphic to $\mathbb{A}^{3}$
equipped with the $\mathbb{C}^{*}$-action
: $\lambda\cdot(x_{2},y,t)=(\lambda^{p}x_{2},\lambda^{-p}y,\lambda t)$.

b) In particular, if $p=1$ then
$X_{2,x_{1},1}$ is isomorphic
to $\mathbb{S}(\tilde{\mathbb{A}}_{(u,v)}^{2},\mathcal{D})$
with $\mathcal{D}=\left\{ 1\right\} D+[0,1]E$.
Since $\mathcal{D}=\left\{ 1\right\} D+[0,1]E$
is equivalent to $\mathcal{D}'=[-1,0]E$,
we have that $X_{2,x_{1},1}$ is
equivariantly isomorphic to $\mathbb{S}(\tilde{\mathbb{A}}_{(u,v)}^{2},\mathcal{D}')$. 
\end{example}
\noindent
\begin{example}
\label{3.2}Choosing $n=2$ and
$f(x_{1},x_{2})=x_{1}+(x_{1}^{d}+x_{2}^{d})^{l}$
yields that 
\[
X_{2,p,f}=\{x_{1}+y^{dl-1}(x_{1}^{d}+x_{2}^{d})^{l}+t^{p}=0\}\subset\mathbb{A}^{4}
\]
 is isomorphic to $\mathbb{S}(\tilde{\mathbb{A}}_{(u,v)}^{2},\mathcal{D}=\left\{ \frac{1}{p}\right\} D+[0,\frac{1}{p}]E)$
, where $E$ is the exceptional
divisor of the blow up and $D$
is the strict transform of the
curve $\{v+(v^{d}+u^{d})^{l}=0\}\subset\mathbb{A}^{2}$.
Note that in contrast with the
previous example, $X_{2,p,f}$
is not isomorphic to $\mathbb{A}^{3}$.
Indeed, if it were, then by the
result of Koras-Russell \cite{Ka-K-ML-R},
the $\mathbb{C}^{*}$ action on
$X_{2,p,f}$ would be linearizable.
By considering the linear action
induced on the tangent space of
the fixed point, we find that $X_{2,p,f}$
would have to be equivariantly
isomorphic to $X_{2,x_{1},p}$
for some $p$. On the other hand
it follows from \cite[corollary 8.12]{A-H}
that two pp-divisors $\mathcal{D}_{i}$,
defined on $Y_{i}$ respectively
with the same tail cone, define
equivariantly isomorphic varieties
$\mathbb{S}(Y_{i},\mathcal{D}_{i})$
if and only if there exist projective
birational morphisms $\psi_{i}:Y_{i}\rightarrow Y$
and a pp-divisor $\mathcal{D}$
on $Y$ such that $\mathcal{D}_{i}\simeq\psi_{i}^{*}(\mathcal{D})$
$i=1,2$. This would induce an
automorphism $\phi$ of $\tilde{\mathbb{A}}^{2}$,
such that $\phi^{*}(f)=x_{1}$,
which is not possible, since a
general fiber of $f$ is singular.
\end{example}

\subsection{Koras-Russell threefolds}

Smooth affine, contractible threefolds
with a hyperbolic $\mathbb{C}^{*}$-action
whose quotient is isomorphic to
$\mathbb{A}^{2}/G$ where $G$
is a finite cyclic group have been
classified by Koras and Russell
\cite{K-R}, in the context of
the linearization problem for $\mathbb{C}^{*}$-actions
on $\mathbb{A}^{3}$ \cite{Ka-K-ML-R}.
These threefolds, which we call
Koras-Russell threefolds, provide
examples of $\mathbb{T}$-varieties
of complexity two. According to
\cite{Ka-ML} they admit the following
description:

Let $a'$,$b'$ and $c'$ be pairwise
prime natural numbers with $b'\geq c'$
and let $\mu_{a'}$, the group
of $a'$-th roots of unity, act
on $\mathbb{A}^{2}=\mathrm{Spec}(\mathbb{C}[u,v])$
by $(u,v)\to(\lambda^{c'}u,\lambda^{b'}v)$
where $\lambda\in\mu_{a'}$. Consider
a semi-invariant polynomial $f$
of weight congruent to $b'$ modulo
$a'$ and with the property that
$L=\{f=0\}$ is isomorphic to a
line and meets the axis $u=0$
transversely at the origin and
at $r-1\geq1$ other points. With
these assumptions the polynomial
$s^{-c'}f(s^{c'}u,s^{b'}v)$ can
be rewritten in the form $F(w,u,v)$
with $w=s^{a'}$where $F$ is semi-invariant
of weight $b'$ for the $\mathbb{C}^{*}$-action
$(w,u,v)\longmapsto(\lambda^{-a'}w,\lambda^{c'}u,\lambda^{b'}v)$.
Then for any choice of pairwise
prime integers $(\alpha_{1},\alpha_{2},\alpha_{3})$
such that $\gcd(\alpha_{1},a')=\gcd(\alpha_{2},b')=\gcd(\alpha_{2},c')=1$,
the hypersurface $X=\{(x,y,z,t)\in\mathbb{A}^{4}/t^{\alpha_{3}}+F(y^{\alpha_{1}},z^{\alpha_{2}},x)=0\}$
is a Koras-Russell threefold.

Here we mainly consider two families
of such threefolds: 

1) The first kind is defined by
equations of the form: 
\[
\{x+x^{d}y+z^{\alpha_{^{2}}}+t^{\alpha_{3}}=0\},
\]
 where $2\leq d$, $2\leq\alpha_{2}<\alpha_{3}$
with $\gcd(\alpha_{2},\alpha_{3})=1$
and equipped with the $\mathbb{C}^{*}$-action
induced by the linear one on $\mathbb{A}^{4}$
with weights $(\alpha_{2}\alpha_{3},-(d-1)\alpha_{2}\alpha_{3},\alpha_{3},\alpha_{2})$.
These correspond to the choice
of $f=u+v+v^{d}$.

2) The second type is defined by
\[
\{x+y(x^{d}+z^{\alpha_{^{2}}})^{l}+t^{\alpha_{3}}=0\},
\]
 where $2\leq d$, $1\leq l$ ,
$2\leq\alpha_{2}<\alpha_{3}$ with
$\gcd(\alpha_{2},d)=\gcd(\alpha_{2},\alpha_{3})=1$
and equipped with the $\mathbb{C}^{*}$-action
induced by the linear one on $\mathbb{A}^{4}$
with weights $(\alpha_{2}\alpha_{3},-(dl-1)\alpha_{2}\alpha_{3},d\alpha_{3},\alpha_{2})$.
These correspond to the choice
of $f=v+(u+v^{d})^{l}$.\\

To obtain the Altmann-Hausen representation
for these threefolds, we will exploit
the fact that they arise as $\mathbb{C}^{*}$-equivariant
bi-cyclic covers of $\mathbb{A}^{3}$.
We will see that the polyhedral
coefficients are related with the
choice of $(\alpha_{1},\alpha_{2},\alpha_{3})$
and the divisors are related with
the choice of the fiber $L=\{f=0\}$
in the construction above.

\subsection{The Russell Cubic}

We begin with the Russell cubic
$X=\{x+x^{2}y+z^{2}+t^{3}=0\}$
in $\mathbb{A}^{4}=\mathrm{Spec}(\mathbb{C}[x,y,z,t])$
which corresponds to the choice
$a'=b'=c'=1$, $\alpha_{1}=1$,
$\alpha_{2}=2$, $\alpha_{3}=3$
and $f(u,v)=u+v+v^{2}$ in the
construction above. By construction
$X$ is equiped with the $\mathbb{C}^{*}$-action
induced by the linear one on $\mathbb{A}^{4}$
with weights $(6,-6,3,2)$. The
algebraic quotient $X/\!/\mathbb{C}^{*}$
is isomorphic to $\mathbb{A}_{(u,v)}^{2}=\mathrm{Spec}(\mathbb{C}[u,v])$
where $u=yz^{2}$ and $v=yx$.
\begin{prop}
(see also \cite{I-V}) The Russell
Cubic $X$ is isomorphic to $\mathbb{S}(\tilde{\mathbb{A}}_{(u,v)}^{2},\mathcal{D})$
for 
\[
\mathcal{D}=\left\{ \frac{1}{2}\right\} D_{3}+\left\{ -\frac{1}{3}\right\} D_{2}+\left[0,\frac{1}{6}\right]E,
\]
where $E$ is the exceptional divisor
of $\pi:\tilde{\mathbb{A}}_{(u,v)}^{2}\rightarrow\mathbb{A}^{2}$,
and where $D_{2}$ and $D_{3}$
are the strict transforms of the
curves $\left\{ u=0\right\} $
and $\left\{ u+v+v^{2}=0\right\} $
in $\mathbb{A}^{2}$ respectively.\end{prop}
\begin{proof}
The two projections $\Phi_{2}=\mathrm{pr}_{x,y,t}:X\rightarrow X_{2}=\mathbb{A}^{3}$
and $\Phi_{3}=\mathrm{pr}_{x,y,z}:X\rightarrow X_{3}=\mathbb{A}^{3}$
express $X$ as cyclic Galois covers
of $\mathbb{A}^{3}$ of degrees
$2$ and $3$ respectively, whose
Galois groups $\mu_{2}$ and $\mu_{3}$
act on $X$ by $\xi\cdot(x,y,z,t)=(x,y,\xi z,t)$
and $\zeta\cdot(x,y,z,t)=(x,y,z,\zeta t)$
respectively. Furthermore these
two actions commute and the quotient
$X_{6}=X/\!/(\mu_{2}\times\mu_{3})$
is isomorphic to $\mathbb{A}^{3}=\mathrm{Spec}(\mathbb{C}[x,y,z^{2}])$.
Letting $A=\underset{n\in\mathbb{Z}}{\bigoplus}A_{n}$
be the coordinate ring of $X$
equipped with the grading corresponding
to the given $\mathbb{C}^{*}$-action,
we have in fact $X_{\ell}=\mathrm{Spec}(\underset{n\in\mathbb{Z}}{\bigoplus}A_{\ell n})$,
$\ell=2,3,6$. This yields a $\mathbb{C}^{*}$-equivariant
commutative diagram

\[
\xymatrix{ & X\ar[dr]^{\Phi_{3}}\ar[dl]_{\Phi_{2}}\ar[dd]_{\Phi_{6}}\\
X_{2}=X/\!/\mu_{2}\ar[dr] &  & X_{3}=X/\!/\mu_{3}\ar[dl]\\
 & X_{6}=X/\!/(\mu_{2}\times\mu_{3})
}
\]
where $\mathbb{C}^{*}$ acts linearly
on $X_{2}$, $X_{3}$ and $X_{6}$
with weights $(3,-3,1)$, $(2,-2,1)$
and $(1,-1,1)$ respectively.

Furthermore since the action of
$\mu_{2}\times\mu_{3}$ on $X$
factors through that of $\mathbb{C}^{*}$
we deduce from Theorem \ref{thm:Main_TH}
that $\Phi_{2}$ corresponds to
the map of proper polyhedral divisors
$(\mathrm{id},F_{2},1)$ and $\Phi_{3}$
corresponds to the map of proper
polyhedral divisors $(\mathrm{id},F_{3},1)$
where $F_{\ell}^{*}(\mathcal{D})=\ell\mathcal{D}$,
$\ell=2,3,6$. The semi-projective
varieties $Y(X)$ and $Y(X_{\ell})$,
$\ell=2,3,6$ are all isomorphic.
As observed earlier, $A_{0}=\mathbb{C}[u,v]$
with $u=yz$ and $v=yx$ so that
$Y_{0}(X_{6})\simeq Y_{0}(X)=\mathbb{A}_{(u,v)}^{2}$.
We further observe that $A_{-6n}=A_{0}\cdot y^{n}\subset A$
because all semi-invariant polynomials
of negative weights divisible by
$6$ are divisible by $y$. This
implies that $Y_{-}(X_{6})\simeq\mathrm{Proj}(\underset{n\in\mathbb{Z}_{\geq0}}{\bigoplus}A_{0}.y^{n})\simeq Y_{0}(X)$.
Finally, $\underset{n\in\geq0}{\bigoplus}A_{6n}\simeq\mathrm{Sym}_{A_{0}}A_{6}$
where $A_{6}$ is the free $A_{0}$-submodule
of $A$ generated by $x$ and $z$.
Therefore 

\begin{eqnarray*}
Y(X)\simeq Y(X_{6}) & = & Y_{-}(X_{6})\times_{Y_{0}(X_{6})}Y_{+}(X_{6})\simeq Y_{+}(X_{6})
\end{eqnarray*}
is isomorphic to the blow-up $\tilde{\mathbb{A}}_{(u,v)}^{2}$
of $Y_{0}(X)=\mathbb{A}^{2}$ at
the origin. It remains to determine
the pp-divisor $\mathcal{D}$.
We will construct it from those
$\mathcal{D}_{2}$ and $\mathcal{D}_{3}$
corresponding to $X_{2}$ and $X_{3}$
respectively.

By Proposition \ref{building block},
$X_{2}=\mathbb{S}(\tilde{\mathbb{A}}_{(u,v)}^{2}$,
$\mathcal{D}_{2}=\{\frac{1}{3}\}D_{2}+[0,\frac{1}{3}]E)$
where $D_{2}$ is the strict transform
of the curve $\left\{ u=0\right\} $
and $E$ is the exceptional divisor
and $X_{3}=\mathbb{S}(\tilde{\mathbb{A}}_{(u',v)}^{2}$,
$\mathcal{D}_{3}=\{\frac{1}{2}\}D_{3}+[0,\frac{1}{2}]E)$
where $D_{3}$ is the strict transform
of the curve $\left\{ u'=0\right\} $
and $E$ is the exceptional divisor.
Theorem \ref{thm:Main_TH} implies
in turn that $2\mathcal{D}\sim\mathcal{D}_{2}=\{\frac{1}{3}\}D_{2}+[0,\frac{1}{3}]E$
and $3\mathcal{D}\sim\mathcal{D}_{3}=\{\frac{1}{2}\}D_{3}+[0,\frac{1}{2}]E$.
Thus $\mathcal{D}_{2}+\mathcal{D}=\mathcal{D}_{3}$
and we conclude that $\mathcal{D}=\left\{ \frac{1}{2}\right\} D_{3}+\left\{ -\frac{1}{3}\right\} D_{2}+\left[0,\frac{1}{6}\right]E$
.\end{proof}
\begin{rem*}
The choice of the coefficients
is not unique since $\mathcal{D}'\sim\mathcal{D}+\mathrm{div}(f)$
for any rational function $f$
on $Y$. This corresponds for example
to $\mathcal{D}'\sim\mathcal{D}+D_{3}+E$
and more generally for any pair
$(a,b)\in\mathbb{Z}^{2}$ such
that $3a+2b=1$ we have that $\mathcal{D}\sim\left\{ \frac{a}{2}\right\} D_{3}+\left\{ \frac{b}{3}\right\} D_{2}+\left[0,\frac{1}{6}\right]E$.
\end{rem*}

\subsection{Koras Russell threefolds of the
first kind.}

Now we will show that a similar
method can be used to present all
Koras-Russell threefolds of the
form $X=\{x+x^{d}y+z^{\alpha_{^{2}}}+t^{\alpha_{3}}=0\}$
in $\mathbb{A}^{4}=\mathrm{Spec}(\mathbb{C}[x,y,z,t])$.
Namely, we consider a cyclic cover
$V$ of $X$ with algebraic quotient
$V/\!/\mathbb{C}^{*}$ isomorphic
to $\mathbb{A}^{2}=\mathrm{Spec}(\mathbb{C}[u,v])$
where $u=yz^{\alpha_{2}}$ and
$v=yx$. A representation of $V$
is obtained by the same method
as in the previous case and the
representation of $X$ is deduced
by applying again Theorem \ref{thm:Main_TH}. 

The categorical quotient $X/\!/\mathbb{C}^{*}$
is isomorphic to $\mathbb{A}_{(u,v)}^{2}/\!/\mu_{d-1}$
where $\mu_{d-1}$ acts by $\xi\cdot(u,v)=(\xi u,\xi v)$.
So we consider $V$ a finite cyclic
cover of $X$ given by the equation
$X=\{x+x^{d}y^{d-1}+z^{\alpha_{^{2}}}+t^{\alpha_{3}}=0\}$
in $\mathbb{A}^{4}=\mathrm{Spec}(\mathbb{C}[x,y,z,t])$,
equipped with the $\mathbb{C}^{*}$-action
induced by the linear one on $\mathbb{A}^{4}$
with weights $(\alpha_{2}\alpha_{3},-\alpha_{2}\alpha_{3},\alpha_{3},\alpha_{2})$.
Furthermore $\mu_{\alpha_{2}}\times\mu_{\alpha_{3}}\times\mu_{d-1}$
acts on $V$ by $(\zeta,\epsilon,\xi)\cdot(x,y,z,t)\rightarrow(x,\xi y,\zeta z,\epsilon t)$.
Observe that the action of $\mu_{\alpha_{2}}\times\mu_{\alpha_{3}}$
factors through that of $\mathbb{C}^{*}$.
This yields the following diagram
of quotient morphisms: 
\[
\xymatrix{ & V\ar[dr]^{\Phi_{\alpha_{3}}}\ar[dl]_{\Phi_{\alpha_{2}}}\ar[d]_{\Phi_{\mu_{d-1}}}\\
\mathbb{A}^{3}\simeq V/\!/\mu_{\alpha_{2}} & X=V/\!/\mu_{d-1} & \mathbb{A}^{3}\simeq V/\!/\mu_{\alpha_{3}}.
}
\]

By Theorem \ref{thm:Main_TH},
$\Phi_{\alpha_{2}}$ corresponds
to the map of proper polyhedral
divisors $(\mathrm{id},F_{\alpha_{2}},1)$
and $\Phi_{\alpha_{3}}$ corresponds
to the map of proper polyhedral
divisor $(\mathrm{id},F_{\alpha_{3}},1)$
where $F_{\ell}^{*}(\mathcal{D})=\ell\mathcal{D}$,
$\ell=2,3,6$. In addition we obtain
that $Y(V)$ is isomorphic to the
blow-up $\tilde{\mathbb{A}}_{(u,v)}^{2}$
of $\mathbb{A}^{2}$ at the origin
on which $\mu_{\alpha_{2}}\times\mu_{\alpha_{3}}\times\mu_{d-1}$
acts by $(\zeta,\epsilon,\xi)\cdot(u,v)=(\xi u,\xi v)$.
This leads to the following diagram:

\[
\xymatrix{ & Y(V)\ar[dr]^{\simeq}\ar[dl]_{\simeq}\ar[d]_{\varphi_{\mu_{d-1}}}\\
Y(V_{\alpha_{2}}) & Y(X)\simeq Y(V)/\!/\mu_{d-1} & Y(V_{\alpha_{3}}).
}
\]

Using example \ref{3.1} we obtain
Altmann-Hausen representations
of $V/\!/\mu_{\alpha_{2}}$ and
$V/\!/\mu_{\alpha_{3}}$ in the
form $\mathbb{S}(\tilde{\mathbb{A}}_{(u,v)}^{2}$,
$\mathcal{D}_{\alpha_{2}}=\{\frac{1}{\alpha_{3}}\}D_{\alpha_{2}}+[0,\frac{1}{\alpha_{3}}]E)$
where $D_{\alpha_{2}}$ is the
strict transform of the curve $\left\{ u=0\right\} $,
$E$ is the exceptional divisor
and $\mathbb{S}(\tilde{\mathbb{A}}_{(u',v)}^{2}$,
$\mathcal{D}_{\alpha_{3}}=\{\frac{1}{\alpha_{2}}\}D_{\alpha_{3}}+[0,\frac{1}{\alpha_{2}}]E)$
where $D_{\alpha_{3}}$ is the
strict transform of the curve $\left\{ u'=0\right\} $,
$E$ is the exceptional divisor.
This implies that $V$ is isomorphic
to $\mathbb{S}(\tilde{\mathbb{A}}_{(u,v)}^{2},\mathcal{D})$
for 
\[
\mathcal{D}=\left\{ \frac{a}{\alpha_{2}}\right\} D_{\alpha_{3}}+\left\{ \frac{b}{\alpha_{3}}\right\} D_{\alpha_{2}}+\left[0,\frac{1}{\alpha_{2}\alpha_{3}}\right]E\;(*),
\]
where $E$ is the exceptional divisor
of $\pi:\tilde{\mathbb{A}}_{(u,v)}^{2}\rightarrow\mathbb{A}^{2}$,
$D_{\alpha_{2}}$ and $D_{\alpha_{3}}$
are the strict transforms of the
curves $\left\{ u=0\right\} $
and $\left\{ u+v+v^{d}=0\right\} $
in $\mathbb{A}_{(u,v)}^{2}$ respectively,
and $(a,b)\in\mathbb{Z}^{2}$ are
chosen such that $a\alpha_{3}+b\alpha_{2}=1$.
Applying Theorem \ref{thm:Main_TH}
we obtain 
\begin{prop}
The Koras-Russell threefold $X=\{x+x^{d}y+z^{\alpha_{^{2}}}+t^{\alpha_{3}}=0\}$
in $\mathbb{A}^{4}=\mathrm{Spec}(\mathbb{C}[x,y,z,t])$
is isomorphic to $\mathbb{S}(\tilde{\mathbb{A}}_{(u,v)}^{2}/\!/\mu_{d-1},\mathcal{D}')$
for

\[
\mathcal{D}'=\left\{ \frac{a}{\alpha_{2}}\right\} D'_{\alpha_{3}}+\left\{ \frac{b}{\alpha_{3}}\right\} D'_{\alpha_{2}}+\left[0,\frac{1}{(d-1)\alpha_{2}\alpha_{3}}\right]E'
\]
where $\mathcal{D}=\varphi_{\mu_{d-1}}^{*}(\mathcal{D}')$
, $\mathcal{D}$ is defined in
the relation $(*)$ and $D'_{\alpha_{3}}$,$D'_{\alpha_{2}}$
are prime divisors and $E'$ is
the exceptional divisor of the
blow-up of the singularity in $\mathbb{A}^{2}/\!/\mu_{d-1}$.
\end{prop}

\subsection{Koras Russell threefolds of the
second kind.}

For Koras-Russell threefolds of
the second kind $X=\{x+y(x^{d}+z^{\alpha_{^{2}}})^{l}+t^{\alpha_{3}}=0\}$
in $\mathbb{A}^{4}=\mathrm{Spec}(\mathbb{C}[x,y,z,t])$
the construction will be slightly
different due to the fact that
the variables $z$ and $t$ do
no longer play symmetric roles.
We will consider again a cyclic
cover $V$ of $X$, but in this
case $V/\!/\mu_{\alpha_{2}}$ will
not be isomorphic to $\mathbb{A}^{3}$.
Recall that by definition, $\alpha_{2}$
and $d$ are coprime. We consider
a bi-cyclic cover $V=\{x+y^{dl-1}(x^{d}+z^{d\alpha_{^{2}}})^{l}+t^{\alpha_{3}}=0\}$
of $X$ of order $d\times(dl-1)$,
which we decompose as a cyclic
cover $\phi_{d}:V\rightarrow V_{d}=\{x+y^{dl-1}(x^{d}+z^{\alpha_{^{2}}})^{l}+t^{\alpha_{3}}=0\}$
of degree $d$, followed by a cyclic
cover $\phi_{dl-1}:V_{d}\rightarrow X$
of degree $dl-1$. The hypersurface
$V$ is equipped with the $\mathbb{C}^{*}$-action
induced by the linear one on $\mathbb{A}^{4}$
with weights $(\alpha_{2}\alpha_{3},-\alpha_{2}\alpha_{3},\alpha_{3},\alpha_{2})$
and with the action of $\mu_{\alpha_{2}}\times\mu_{\alpha_{3}}\times\mu_{dl-1}\times\mu_{d}$
defined by $(\zeta,\epsilon,\xi,\delta)\cdot(x,y,z,t)=(x,\xi y,\zeta\delta z,\epsilon t)$.
The action of $\mu_{\alpha_{2}}\times\mu_{\alpha_{3}}$
on $V$ factors through that of
$\mathbb{C}^{*}$ and we obtain
the following diagram: 

\[
\xymatrix{ & V\ar[dr]^{\Phi_{\alpha_{3}}}\ar[dl]_{\Phi_{\alpha_{2}}}\ar[d]_{\Phi_{\mu_{d}}}\\
V_{\alpha_{2}}=V/\!/\mu_{\alpha_{2}} & V_{d}=V/\!/\mu_{d}\ar[d]_{\Phi_{\mu_{dl-1}}} & \mathbb{A}^{3}\simeq V_{\alpha_{3}}=V/\!/\mu_{\alpha_{3}}\\
 & X=V/\!/(\mu_{d}\times\mu_{dl-1}) & .
}
\]

By Theorem \ref{thm:Main_TH},
considering $\Phi_{\alpha_{3}}$,
we obtain that $Y(V)$ is isomorphic
to the blow-up $\tilde{\mathbb{A}}_{(u,v)}^{2}$
of $\mathbb{A}^{2}$ where $u=yz^{\alpha_{2}}$
and $v=yx$ on which $\mu_{\alpha_{2}}\times\mu_{\alpha_{3}}\times\mu_{dl-1}\times\mu_{d}$
acts by $(\zeta,\epsilon,\xi,\delta)\cdot(u,v)=(\xi\delta^{\alpha_{^{2}}}u,\xi v)$.
We obtain the following quotient
diagram:

\[
\xymatrix{ & Y(V)\ar[dr]^{\simeq}\ar[dl]_{\simeq}\ar[d]_{\varphi_{\mu_{d}}}\\
Y(V_{\alpha_{2}}) & Y(V_{d})\ar[d]_{\varphi_{\mu_{dl-1}}} & Y(V_{\alpha_{3}})\\
 & Y(X) & .
}
\]

Now by Proposition \ref{building block}
$V_{\alpha_{2}}=\mathbb{S}(\tilde{\mathbb{A}}_{(u,v)}^{2}$,
$\mathcal{D}_{\alpha_{2}}=\{\frac{1}{\alpha_{3}}\}D_{\alpha_{2}}+[0,\frac{1}{\alpha_{3}}]E)$
where $D_{\alpha_{2}}$ is the
strict transform of the curve $\left\{ v+(v{}^{d}+u^{d})^{l})=0\right\} $
and $E$ is the exceptional divisor,
and $V_{\alpha_{3}}=\mathbb{S}(\tilde{\mathbb{A}}_{(u,v)}^{2}$,
$\mathcal{D}_{\alpha_{3}}=\{\frac{1}{\alpha_{2}}\}D_{\alpha_{3}}+[0,\frac{1}{\alpha_{2}}]E)$
where $D_{\alpha_{3}}$ is the
strict transform of the curve $\left\{ u=0\right\} $
and $E$ is the exceptional divisor.
Thus $V=\mathbb{S}(\tilde{\mathbb{A}}_{(u,v)}^{2},\mathcal{D})$
for
\[
\mathcal{D}=\left\{ \frac{a}{\alpha_{2}}\right\} D_{\alpha_{3}}+\left\{ \frac{b}{\alpha_{3}}\right\} D_{\alpha_{2}}+\left[0,\frac{1}{\alpha_{2}\alpha_{3}}\right]E,
\]
where $E$ is the exceptional divisor
of $\pi:\tilde{\mathbb{A}}_{(u,v)}^{2}\rightarrow\mathbb{A}^{2}$,
and where $D_{\alpha_{2}}$ and
$D_{\alpha_{3}}$ are the respective
strict transforms of the curves
$\left\{ v+(v{}^{d}+u^{d})^{l})=0\right\} $
and $\left\{ u=0\right\} $ in
$\mathbb{A}^{2}$ and $(a,b)\in\mathbb{Z}^{2}$
$a\alpha_{3}+b\alpha_{2}=1$. Note
that the choice of $\mathcal{D}$
up to linear equivalence does not
depend of the choice on $(a,b)\in\mathbb{Z}^{2}$
. 

Now we deduce from Theorem \ref{thm:Main_TH}
that $V_{d}=\mathbb{S}(\tilde{\mathbb{A}}_{(u',v'^{d})}^{2},\mathcal{D}_{d})$
for 
\[
\mathcal{D}_{d}=\left\{ \frac{a'}{\alpha_{2}}\right\} D_{d,\alpha_{3}}+\left\{ \frac{b'}{\alpha_{3}}\right\} D_{d,\alpha_{2}}+\left[0,\frac{1}{\alpha_{2}\alpha_{3}}\right]E_{d}\;(**),
\]
where $a'=a/d$, $b'=b$, $E_{d}$
is the exceptional divisor of $\pi:\tilde{\mathbb{A}}_{(u',v'^{d})}^{2}\rightarrow\mathbb{A}^{2}$
due to the fact that $\tilde{\mathbb{A}}_{(u',v')}^{2}/\!/\mu_{d}\simeq\tilde{\mathbb{A}}_{(u',v'^{d})}^{2}$
for the action of $\mu_{d}$ as
above, and where $D_{d,\alpha_{2}}$
and $D_{d,\alpha_{3}}$ are the
strict transforms of the curves
$\left\{ v'+(u'+v'^{d})^{l})=0\right\} $
and $\left\{ u'=0\right\} $ (
$u'=\varphi_{d}(u^{d})$) in $\mathbb{A}^{2}=\mathrm{Spec}(\mathbb{C}[u',v'])$
respectively. Applying again Theorem
\ref{thm:Main_TH} we obtain :
\begin{prop}
A Koras-Russell threefold $X=\{x+y(x^{d}+z^{\alpha_{^{2}}})^{l}+t^{\alpha_{3}}=0\}$
in $\mathbb{A}^{4}=\mathrm{Spec}(\mathbb{C}[x,y,z,t])$
is isomorphic to $\mathbb{S}(\tilde{\mathbb{A}}_{(u',v'^{d})}^{2}/\!/\mu_{dl-1},\mathcal{D}_{d(dl-1)})$
for

\[
\mathcal{D}_{d(dl-1)}=\left\{ \frac{a'}{\alpha_{2}}\right\} D_{d(dl-1),\alpha_{3}}+\left\{ \frac{b'}{\alpha_{3}}\right\} D_{d(dl-1),\alpha_{2}}+\left[0,\frac{1}{(dl-1)\alpha_{2}\alpha_{3}}\right]E_{d(dl-1)},
\]
where $\mathcal{D}_{d}=\varphi_{\mu_{dl-1}}^{*}(\mathcal{D}_{d(dl-1)})$
, $\mathcal{D}_{d}$ is defined
in the relation $(**)$ and $E_{d(dl-1)}$
is the exceptional divisor of the
blow-up of the singularity in $\mathbb{A}^{2}/\!/\mu_{dl-1}$.\end{prop}

\end{document}